\documentclass[12pt,reqno]{amsart}

\usepackage{mathtools}
\mathtoolsset{showonlyrefs}
\numberwithin{equation}{section}
\usepackage{amssymb}
\usepackage{amsthm}
\usepackage{longtable}
\usepackage{array}
\usepackage{ragged2e}
\usepackage{booktabs}
\usepackage{microtype}
\usepackage[hidelinks]{hyperref}

\theoremstyle{plain}
\newtheorem{lemma}{Lemma}[section]
\newtheorem{prop}[lemma]{Proposition}
\newtheorem{thm}[lemma]{Theorem}

\newtheorem{rem}[lemma]{Remark}

\newcommand{\Fp}{\mathbb{F}_p}
\newcommand{\Fpsq}{\mathbb{F}_{p^2}}

\newcommand{\Dp}{D_p}
\newcolumntype{P}[1]{>{\RaggedRight\arraybackslash}p{#1}}

\begin{document}

\title[modulo eight hypergeometric polynomials]{The modulo eight behavior of common factors of two hypergeometric polynomials arising from a genus-three hyperelliptic family}
\author{Youhei MORITA}

\address{Graduate School of Information Science and Technology,
The University of Osaka, 1-5 Yamadaoka, Suita, Osaka 565-0871, Japan}
\email{morita.youhei@ist.osaka-u.ac.jp}


\begin{abstract}
We study a one-parameter family of smooth projective genus-three hyperelliptic curves over algebraically closed fields of odd characteristic, with the parameter different from zero and one.
We translate superspeciality into the simultaneous vanishing of two truncated hypergeometric polynomials arising from the Cartier operator.  For their monic greatest common divisor,
we prove squarefreeness and show that every irreducible factor has degree at most two.
We then determine its behavior according to the characteristic modulo eight:
in one residue class every nonconstant irreducible factor is quadratic;
in two residue classes the greatest common divisor is trivial;
and in the remaining residue class a distinguished linear factor always occurs.
\end{abstract}

\maketitle
\tableofcontents

\section{Introduction}\label{sec:introduction}

Let $p>2$ be a prime, and let $k$ be an algebraically closed field of characteristic $p$. We consider the smooth projective genus-three hyperelliptic curve
\begin{align*}
 C_\alpha : y^2=(x^4-1)(x^4-\alpha), \qquad \alpha\in k\setminus\{0,1\}.
\end{align*}
We call a curve superspecial if its Jacobian is isomorphic over $k$ to a product of supersingular elliptic curves.
Our aim is to study the superspecial members of this family through an explicit Cartier--Manin calculation and to extract arithmetic information from the resulting polynomial conditions.

We put
\begin{align*}
 n:=\frac{p-1}{2},\qquad s:=\left\lfloor\frac{p}{4}\right\rfloor,
\end{align*}
and define
\begin{align*}
 \Phi_p(x)&:=\sum_{i=0}^{n}\binom{n}{i}^{2}x^i,\\
 G_p(x)&:=\sum_{i=0}^{s}\binom{n}{i}\binom{n}{s-i}x^i.
\end{align*}
We regard these polynomials as elements of $\Fp[x]$ by reducing their coefficients modulo $p$.
These are truncated hypergeometric polynomials. More precisely,
\begin{align*}
 \Phi_p(x)&={}_2F_1(-n,-n;1;x),\\
 G_p(x)&=\binom{n}{s}\,{}_2F_1(-n,-s;n-s+1;x).
\end{align*}
For $\lambda\in k\setminus\{0,1\}$, let
\begin{align*}
 E_\lambda:\quad y^2=x(x-1)(x-\lambda)
\end{align*}
denote the Legendre elliptic curve. 
Up to the nonzero scalar $(-1)^n$, the polynomial $\Phi_p$ is the classical Deuring polynomial for this family, and its roots are precisely the supersingular Legendre parameters (see \cite{Igusa}, \cite{AuerTop}).
For background on supersingular and superspecial loci (see \cite{Oort}, \cite{KatsuraOort}, \cite{IbukiyamaKatsuraOort}, \cite{GeerVlugt}), related explicit investigations of low-genus hyperelliptic families appear in \cite{Ohashi}.

Ohashi and Kudo \cite{OhashiKudo} gave hypergeometric expressions for the Cartier--Manin coefficients of the more general family $y^2=(x^r-1)(x^r-\lambda)$.
The calculation in Section~\ref{sec:superspecial-criterion} specializes that framework to $r=4$ and isolates the two polynomials whose common zeros describe the superspecial locus of the present family.
Our congruence argument uses the more specialized standard-form statement in \cite[Lemma~4.2]{Ohashi}.

We define $D_p(x)$ to be the \emph{monic} greatest common divisor
\begin{align*}
 D_p(x):=\gcd\bigl(\Phi_p(x),G_p(x)\bigr)\in\Fp[x].
\end{align*}
The next proposition identifies superspecial members of the family with the common zeros of $\Phi_p$ and $G_p$. 
This geometric interpretation will play a central role throughout the paper and will be used repeatedly in the proof of the main theorem (see Proposition~\ref{prop:root-locus}).

\begin{prop}\label{prop:superspecial-equivalence}
For $\alpha\in k\setminus\{0,1\}$, the following conditions are equivalent:
\begin{enumerate}
 \item $C_\alpha$ is superspecial.
 \item $\Phi_p(\alpha)=G_p(\alpha)=0$.
\end{enumerate}
\end{prop}

The proof is given at the end of Section~\ref{sec:superspecial-criterion}, after the Cartier--Manin calculation has been solved.

The following theorem summarizes the main arithmetic consequences established in this paper.

\begin{thm}\label{thm:main}
Let $p>2$ be a prime. Then:
\begin{enumerate}
 \item $D_p$ is squarefree, and every nonconstant irreducible factor of $D_p$ over $\Fp$ has degree at most $2$.
 \item If $p\equiv1\pmod8$, every nonconstant irreducible factor of $D_p$ has degree exactly $2$.
 \item If $p\equiv3$ or $5\pmod8$, then $D_p(x)=1$.
 \item If $p\equiv7\pmod8$, then $x+1$ divides $D_p(x)$.
\end{enumerate}
\end{thm}

The remainder of the paper is organized as follows.
Section~\ref{sec:superspecial-criterion} proves Proposition~\ref{prop:superspecial-equivalence} by translating superspeciality into the common-zero condition for $\Phi_p$ and $G_p$.
Section~\ref{sec:main-theorem} identifies the roots of $D_p$ with superspecial parameters, proves the mod eight obstruction for such parameters, and derives all four parts of Theorem~\ref{thm:main}.
Appendix~\ref{app:computations} contains the factorization table for odd primes below $1000$.

\section{Superspeciality and common zeros}\label{sec:superspecial-criterion}

The purpose of this section is to establish the geometric--polynomial bridge used throughout the paper:
\begin{align}\label{eq:superspecial-common-zero}
 C_\alpha\text{ is superspecial}
 \quad\Longleftrightarrow\quad
 \Phi_p(\alpha)=G_p(\alpha)=0.
\end{align}
Set
\begin{align*}
 f_\alpha(x):=(x^4-1)(x^4-\alpha)=x^8-(1+\alpha)x^4+\alpha,
\end{align*}
and define the coefficients $c_t(\alpha)$ by
\begin{align*}
 f_\alpha(x)^n=\sum_{t\geq0}c_t(\alpha)x^t.
\end{align*}
Superspeciality is equivalent to the vanishing of the Cartier operator.
It remains to express that vanishing condition in terms of the coefficients of $f_\alpha(x)^n$ and then to identify the relevant coefficients with evaluations of $\Phi_p$ and $G_p$.
Although the parity of $n$ changes the positions of the nonzero entries in the Cartier--Manin matrix, the final criterion is uniform for every odd prime $p$.

Let $\tau:k\to k$ denote the inverse of the Frobenius automorphism.
The differentials
\begin{align*}
 \omega_1=\frac{dx}{y},\qquad
 \omega_2=\frac{x\,dx}{y},\qquad
 \omega_3=\frac{x^2\,dx}{y}
\end{align*}
form an ordered basis of $H^0(C_\alpha,\Omega^1_{C_\alpha})$.
Via this basis, we identify $H^0(C_\alpha,\Omega^1_{C_\alpha})$ with the space $k^3$ of column vectors.
With this convention, the Cartier operator is represented by
\begin{align*}
 B(\alpha):=\bigl(\tau(c_{ip-j}(\alpha))\bigr)_{1\leq i,j\leq3}
\end{align*}
(see \cite[Section~3.1]{AchterHowe}).  Consequently,
\begin{align*}
 B(\alpha)=0
 \quad\Longleftrightarrow\quad
 M(\alpha):=\bigl(c_{ip-j}(\alpha)\bigr)_{1\leq i,j\leq3}=0.
\end{align*}
We work with $M(\alpha)$ below, since only its vanishing is relevant.

Let $T$ be an indeterminate, and set
\begin{align*}
 g_T(z):=(z-1)(z-T)=z^2-(1+T)z+T.
\end{align*}
For $0\leq m\leq 2n$, define the coefficient polynomial
$A_m(T)\in\Fp[T]$ by the identity
\begin{align}\label{eq:def-Am}
 g_T(z)^n=\sum_{m=0}^{2n}A_m(T)z^m
 \qquad\text{in }\Fp[T,z].
\end{align}
For $\alpha\in k$, we denote by $A_m(\alpha)$ the value of
$A_m(T)$ at $T=\alpha$. Thus, setting
\begin{align*}
 g_\alpha(z):=g_T(z)\big|_{T=\alpha}
 =(z-1)(z-\alpha),
\end{align*}
we obtain
\begin{align*}
 g_\alpha(z)^n
 =\sum_{m=0}^{2n}A_m(\alpha)z^m.
\end{align*}
Since $f_\alpha(x)=g_\alpha(x^4)$, one has
\begin{align*}
 c_{4m}(\alpha)=A_m(\alpha),\qquad
 c_t(\alpha)=0\quad\text{if }4\nmid t.
\end{align*}

\begin{lemma}\label{lem:reciprocity-A}
Regard the coefficients $A_m$ as polynomials in an indeterminate $T$. 
For
$0\leq m\leq2n$, one has
\begin{align*}
 T^mA_m(T)=T^nA_{2n-m}(T)
 \qquad\text{in }\Fp[T].
\end{align*}
Consequently, for every $a\in k^\times$,
\begin{align*}
 A_m(a)=a^{n-m}A_{2n-m}(a).
\end{align*}
\end{lemma}

\begin{proof}
Set $g_T(z):=(z-1)(z-T)$.  By the definition of the coefficient polynomials $A_m(T)$, one has
$g_T(z)^n=\sum_{m=0}^{2n}A_m(T)z^m$.  The identity
\begin{align*}
 z^2g_T(T/z)=Tg_T(z)
\end{align*}
implies
\begin{align*}
 z^{2n}g_T(T/z)^n=T^ng_T(z)^n.
\end{align*}
Comparing the coefficient of $z^m$ gives
\begin{align*}
 T^nA_m(T)=T^{2n-m}A_{2n-m}(T).
\end{align*}
If $m\leq n$, cancellation of $T^{n-m}$ in the integral domain $\Fp[T]$ gives the claimed identity.
If $m\geq n$, multiplication by $T^{m-n}$ gives the same identity. Evaluation at $T=a\neq0$ yields the second assertion.
\end{proof}

\begin{lemma}\label{lem:matrix-reduction}
For $\alpha\in k^\times$, the Cartier--Manin matrix satisfies
\begin{align*}
 M(\alpha)=0
 \quad\Longleftrightarrow\quad
 A_s(\alpha)=A_n(\alpha)=0.
\end{align*}
\end{lemma}

\begin{proof}
Because only exponents divisible by $4$ occur in $f_\alpha(x)^n$, the parity of $n$ determines the sparsity pattern of $M(\alpha)$.
Since $s=\lfloor(2n+1)/4\rfloor$, a direct reduction of the indices $ip-j$ modulo $4$ gives
\begin{align*}
 M(\alpha)=
 \begin{cases}
 \begin{pmatrix}
  A_s(\alpha)&0&0\\
  0&A_n(\alpha)&0\\
  0&0&A_{2n-s}(\alpha)
 \end{pmatrix},& n\text{ even},\\[2.2em]
 \begin{pmatrix}
  0&0&A_s(\alpha)\\
  0&A_n(\alpha)&0\\
  A_{2n-s}(\alpha)&0&0
 \end{pmatrix},& n\text{ odd}.
 \end{cases}
\end{align*}
Lemma~\ref{lem:reciprocity-A} gives
\begin{align*}
 A_s(\alpha)=\alpha^{n-s}A_{2n-s}(\alpha).
\end{align*}
Since $\alpha\neq0$, the two outer entries vanish simultaneously.  In either parity, therefore, $M(\alpha)$ vanishes exactly when both $A_s(\alpha)$ and $A_n(\alpha)$ vanish.
\end{proof}

\begin{prop}
\label{prop:polynomial-criterion}
For $\alpha\in k\setminus\{0,1\}$, the following conditions are equivalent:
\begin{enumerate}
 \item the Cartier operator on $H^0(C_\alpha,\Omega^1_{C_\alpha})$ is zero;
 \item $\Phi_p(\alpha)=G_p(\alpha)=0$.
\end{enumerate}
\end{prop}

\begin{proof}
For $0\leq m\leq n$, expanding
$g_\alpha(z)^n=(z-1)^n(z-\alpha)^n$ gives
\begin{align*}
 A_m(\alpha)=(-1)^m\alpha^{n-m}
 \sum_{i=0}^{m}\binom{n}{i}\binom{n}{m-i}\alpha^i.
\end{align*}
At the two indices singled out by Lemma~\ref{lem:matrix-reduction}, this becomes
\begin{align*}
 A_n(\alpha)&=(-1)^n\Phi_p(\alpha),\\
 A_s(\alpha)&=(-1)^s\alpha^{n-s}G_p(\alpha).
\end{align*}
Because $\alpha\neq0$, Lemma~\ref{lem:matrix-reduction} yields
\begin{align*}
 M(\alpha)=0
 \quad\Longleftrightarrow\quad
 \Phi_p(\alpha)=G_p(\alpha)=0.
\end{align*}
Finally, $B(\alpha)=0$ if and only if $M(\alpha)=0$, so this is equivalent to the vanishing of the Cartier operator.
\end{proof}

\begin{proof}[Proof of Proposition~\ref{prop:superspecial-equivalence}]
$C_\alpha$ is superspecial if and only if the Cartier operator on $H^0(C_\alpha,\Omega^1_{C_\alpha})$ is zero (see \cite[Theorem~4.1]{Nygaard} and also \cite[Section~5.4]{AchterHowe}).
Proposition~\ref{prop:polynomial-criterion} identifies the latter condition with $\Phi_p(\alpha)=G_p(\alpha)=0$.
\end{proof}

\section{Main Theorem}\label{sec:main-theorem}

Section~\ref{sec:superspecial-criterion} shows that the common roots of $\Phi_p$ and $G_p$ are precisely the parameters for which $C_\alpha$ is superspecial.
We now use this interpretation to prove Theorem~\ref{thm:main}.
The argument has a simple structure.  First we identify the root locus of $D_p$ and obtain the general factorization statement.
Next we prove that a superspecial member of the family can occur only when $p\equiv1$ or $7\pmod8$.
The four assertions of the main theorem then follow by considering the residue class of $p$ modulo eight.

\subsection{The root locus}

The following proposition packages the geometric content of Proposition~\ref{prop:superspecial-equivalence} together with the standard arithmetic properties of the Deuring polynomial.
In particular, it proves Theorem~\ref{thm:main}(1).

\begin{prop}
\label{prop:root-locus}
For $\alpha\in k$, the following are equivalent:
\begin{enumerate}
 \item $D_p(\alpha)=0$;
 \item $\alpha\in\Fpsq\setminus\{0,1\}$ and $C_\alpha$ is superspecial.
\end{enumerate}
Moreover, every root of $D_p$ is simple.  Consequently, $D_p$ is squarefree,
and every nonconstant irreducible factor of $D_p$ over $\Fp$ has degree at
most $2$.
\end{prop}

\begin{proof}
First,
\begin{align*}
 \Phi_p(0)=1.
\end{align*}
Vandermonde's identity also gives
\begin{align*}
 \Phi_p(1)
   =\sum_{i=0}^{n}\binom{n}{i}^2
   =\binom{2n}{n}
   =\binom{p-1}{n}
   \equiv(-1)^n\pmod p.
\end{align*}
Thus neither $0$ nor $1$ is a root of $D_p$.

Let $\alpha\in k\setminus\{0,1\}$.  By the definition of $D_p$ and
Proposition~\ref{prop:superspecial-equivalence},
\begin{align*}
 D_p(\alpha)=0
 \quad\Longleftrightarrow\quad
 \Phi_p(\alpha)=G_p(\alpha)=0
 \quad\Longleftrightarrow\quad
 C_\alpha\text{ is superspecial}.
\end{align*}
If these conditions hold, then $\alpha$ is a root of $\Phi_p$.  By
\cite[Proposition~2.2]{AuerTop}, every supersingular Legendre parameter belongs to $\Fpsq$.
The roots of the Deuring polynomial are simple; this classical separability property is proved in \cite{Igusa} and is recalled by Auer and
Top in the discussion following \cite[Proposition~3.1]{AuerTop}.
Since their Deuring polynomial is $H_p=(-1)^n\Phi_p$, this proves the equivalence and the simplicity assertion.
The final statement follows because every root of $D_p$ lies in $\Fpsq$.
\end{proof}

\subsection{The congruence obstruction}

We now establish the arithmetic restriction that drives the nonexistence part of the main theorem.
The proof begins with a superspecial parameter $\alpha$, places the corresponding curve in the standard form used by Ohashi \cite{Ohashi}, and compares two square conditions with the eighth-power condition of Auer--Top\cite{AuerTop}.

\begin{prop}
\label{prop:congruence-obstruction}
If $C_\alpha$ is superspecial for some
$\alpha\in k\setminus\{0,1\}$, then
\begin{align*}
 p\equiv1\quad\text{or}\quad7\pmod8.
\end{align*}
\end{prop}

\begin{proof}
By Proposition~\ref{prop:superspecial-equivalence}, one has
$\Phi_p(\alpha)=0$.  Thus $\alpha$ is a supersingular Legendre parameter.
Auer and Top prove that
\begin{align*}
 \alpha\in\Fpsq,
 \qquad
 -\alpha\in(\Fpsq^\times)^8;
\end{align*}
see \cite[Propositions~2.2 and~3.1]{AuerTop}.  Choose
$u\in\Fpsq^\times$ such that $u^8=-\alpha$.

Since $8\mid p^2-1$, the cyclic group $\Fpsq^\times$ contains an element
$\varepsilon$ satisfying $\varepsilon^4=-1$.  Put
\begin{align*}
 \beta:=\varepsilon u^2.
\end{align*}
Then $\beta\in\Fpsq^\times$ and $\beta^4=\alpha$.  Choose
$\gamma\in k$ with $\gamma^2=\beta$.  The change of variables
\begin{align*}
 x=\gamma X,
 \qquad
 y=\beta^2Y
\end{align*}
identifies $C_\alpha$ with
\begin{align*}
 Y^2
 &=X^8-(\beta^2+\beta^{-2})X^4+1\\
 &=(X^4-aX^2+1)(X^4+aX^2+1),
\end{align*}
where
\begin{align*}
 a:=\beta+\beta^{-1}\in\Fpsq.
\end{align*}
the two parameters of this standard form are
\begin{align*}
 a_0:=a,\qquad b_0:=-a.
\end{align*}
They belong to $\Fpsq$.
Moreover, the nonsingularity hypotheses in \cite[Lemma~4.2]{Ohashi} are satisfied. Indeed, $a=2$ or $a=-2$ would imply $\beta=1$ or $\beta=-1$, respectively, and hence $\alpha=\beta^4=1$.
If $a=0$, then $\beta^2=-1$, again giving $\alpha=1$. Thus $a_0,b_0\neq\pm2$ and $a_0\neq b_0$. Since the standard-form curve is superspecial, \cite[Lemma~4.2]{Ohashi} applies. In the notation of that lemma, the element $\alpha_+$ satisfying
\begin{align*}
 \alpha_+^2=a_0+2
\end{align*}
belongs to $\Fpsq$. Hence $a+2$ is a square in $\Fpsq$. Since
\begin{align*}
 a+2=\frac{(\beta+1)^2}{\beta}
\end{align*}
and $\beta\neq-1$ (otherwise $\alpha=1$), one has
$\alpha_+\neq0$ and
\begin{align*}
 \beta=\left(\frac{\beta+1}{\alpha_+}\right)^2.
\end{align*}
Thus $\beta$ is a square in $\Fpsq$. 
One has $\beta=v^2$ with $v\in\Fpsq^\times$.
Then
\begin{align*}
 \alpha=\beta^4=v^8.
\end{align*}
Together with $-\alpha=u^8$, this yields
\begin{align*}
 -1=(u/v)^8.
\end{align*}
Put $w:=u/v$.  Then $w^{16}=1$ but $w^8=-1\neq1$, so $w$ has order $16$.
Hence $16\mid p^2-1$.  For an odd prime $p$, this is equivalent to $p\equiv1$ or $7\pmod8$.
\end{proof}

\subsection{Proof of the main theorem}

We can now read off the four conclusions of Theorem~\ref{thm:main}.  Part~\textup{(1)} comes from the root locus. 
Part~\textup{(2)} uses the absence of $\Fp$-rational roots of $\Phi_p$ when $p\equiv1\pmod4$, 
Part~\textup{(3)} follows from the congruence obstruction, and Part~\textup{(4)} is obtained by evaluating the two polynomials at $x=-1$.

\begin{proof}[Proof of Theorem~\ref{thm:main}]
Part~\textup{(1)} is the final assertion of Proposition~\ref{prop:root-locus}.

For Part~\textup{(2)}, suppose that $p\equiv1\pmod8$.  By Part~\textup{(1)}, every irreducible factor of $D_p$ has degree at most $2$.  
Auer and Top \cite{AuerTop} prove that $\Phi_p$ has no $\Fp$-rational root when $p\equiv1\pmod4$ (see \cite[Proposition~3.2(a)]{AuerTop}).
Since $D_p$ divides $\Phi_p$, it has no linear factor.
Thus every nonconstant irreducible factor of $D_p$ is quadratic.

For Part~\textup{(3)}, assume that $p\equiv3$ or $5\pmod8$.  If $D_p$ were nonconstant, it would have a root $\alpha\in k$.
By Proposition~\ref{prop:root-locus}, the curve $C_\alpha$ would be superspecial, contradicting Proposition~\ref{prop:congruence-obstruction}.  Hence $D_p$ is constant, and its monic normalization gives $D_p=1$.

It remains to prove Part~\textup{(4)}.  The value
\begin{align*}
 \Phi_p(-1)=\sum_{i=0}^{n}(-1)^i\binom{n}{i}^2
\end{align*}
is the coefficient of $z^n$ in
\begin{align*}
 (1-z)^n(1+z)^n=(1-z^2)^n.
\end{align*}
It therefore vanishes when $n$ is odd.  Similarly,
\begin{align*}
 G_p(-1)=\sum_{i=0}^{s}(-1)^i\binom{n}{i}\binom{n}{s-i}
\end{align*}
is the coefficient of $z^s$ in $(1-z^2)^n$, and hence vanishes when $s$ is odd.
If $p\equiv7\pmod8$, write $p=8m+7$.
Then
\begin{align*}
 n=\frac{p-1}{2}=4m+3,
 \qquad
 s=\left\lfloor\frac p4\right\rfloor=2m+1,
\end{align*}
so both $n$ and $s$ are odd.
Thus $\Phi_p(-1)=G_p(-1)=0$, and therefore $x+1$ divides $D_p$.
\end{proof}

\begin{rem}\label{rem:scope}
Theorem~\ref{thm:main}(2) does not assert that $D_p$ is nonconstant: $D_p$ may equal $1$ even when $p\equiv1\pmod8$; for example, $D_{113}(x)=1$. 
If $p\equiv7\pmod8$, Theorem~\ref{thm:main}(1) shows that the factors besides $x+1$ are also distinct linear or irreducible quadratic factors.
Both types occur.
For $p=31$, all factors other than $x+1$ are linear, whereas for $p=151$ the remaining factors are irreducible quadratics for $p=47$, both types occur simultaneously (see Appendix~\ref{app:computations}).
\end{rem}

\section*{Acknowledgements}

The author is deeply grateful to Emeritus Professor Toshiyuki Katsura at The University of Tokyo for suggesting the problem that led to this research. The author also sincerely thanks Associate Professor Yasuhiro Wakabayashi at The University of Osaka for courteous comments, patient guidance, and continuous encouragement.

\appendix
\section{Factorization table for odd primes below 1000}\label{app:computations}

The following table gives the factorization of $D_p$ over $\Fp$ for every odd prime $p<1000$.
Every displayed nonconstant factor is monic and irreducible in $\Fp[x]$.
Each coefficient is represented by the unique integer in $\{0,1,\ldots,p-1\}$.
In particular, the table makes the squarefreeness and the degree bound in Theorem~\ref{thm:main}(1) directly visible.
The table is not used in any proof above.

\begingroup
\scriptsize
\renewcommand{\arraystretch}{1.16}
\setlength{\tabcolsep}{3.5pt}
\begin{longtable}{r r r P{0.73\textwidth}}
\caption{Factorization of $\Dp$ over $\Fp$ for all odd primes $p<1000$.}
\label{tab:Dp}\\
\toprule
$p$ & $p\bmod 8$ & $\deg \Dp(x)$ & Factorization of $\Dp(x)$ over $\Fp$\\
\midrule
\endfirsthead
\toprule
$p$ & $p\bmod 8$ & $\deg \Dp(x)$ & Factorization of $\Dp(x)$ over $\Fp$\\
\midrule
\endhead
\midrule
\multicolumn{4}{r}{Continued on next page}\\
\endfoot
\bottomrule
\endlastfoot
3 & 3 & 0 & $ 1 $ \\
5 & 5 & 0 & $ 1 $ \\
7 & 7 & 1 & $ \bigl(x + 1\bigr) $ \\
11 & 3 & 0 & $ 1 $ \\
13 & 5 & 0 & $ 1 $ \\
17 & 1 & 2 & $ \bigl(x^{2} + 14x + 1\bigr) $ \\
19 & 3 & 0 & $ 1 $ \\
23 & 7 & 1 & $ \bigl(x + 1\bigr) $ \\
29 & 5 & 0 & $ 1 $ \\
31 & 7 & 3 & $ \bigl(x + 1\bigr)\allowbreak\, \bigl(x + 17\bigr)\allowbreak\, \bigl(x + 11\bigr) $ \\
37 & 5 & 0 & $ 1 $ \\
41 & 1 & 2 & $ \bigl(x^{2} + 40x + 1\bigr) $ \\
43 & 3 & 0 & $ 1 $ \\
47 & 7 & 7 & $ \bigl(x + 1\bigr)\allowbreak\, \bigl(x + 26\bigr)\allowbreak\, \bigl(x + 30\bigr)\allowbreak\, \bigl(x + 38\bigr)\allowbreak\, \bigl(x + 11\bigr)\allowbreak\, \bigl(x^{2} + 46x + 1\bigr) $ \\
53 & 5 & 0 & $ 1 $ \\
59 & 3 & 0 & $ 1 $ \\
61 & 5 & 0 & $ 1 $ \\
67 & 3 & 0 & $ 1 $ \\
71 & 7 & 1 & $ \bigl(x + 1\bigr) $ \\
73 & 1 & 6 & $ \bigl(x^{2} + 43x + 37\bigr)\allowbreak\, \bigl(x^{2} + 13x + 2\bigr)\allowbreak\, \bigl(x^{2} + 35x + 1\bigr) $ \\
79 & 7 & 3 & $ \bigl(x + 1\bigr)\allowbreak\, \bigl(x + 77\bigr)\allowbreak\, \bigl(x + 39\bigr) $ \\
83 & 3 & 0 & $ 1 $ \\
89 & 1 & 2 & $ \bigl(x^{2} + 76x + 1\bigr) $ \\
97 & 1 & 2 & $ \bigl(x^{2} + 89x + 1\bigr) $ \\
101 & 5 & 0 & $ 1 $ \\
103 & 7 & 1 & $ \bigl(x + 1\bigr) $ \\
107 & 3 & 0 & $ 1 $ \\
109 & 5 & 0 & $ 1 $ \\
113 & 1 & 0 & $ 1 $ \\
127 & 7 & 1 & $ \bigl(x + 1\bigr) $ \\
131 & 3 & 0 & $ 1 $ \\
137 & 1 & 2 & $ \bigl(x^{2} + 86x + 1\bigr) $ \\
139 & 3 & 0 & $ 1 $ \\
149 & 5 & 0 & $ 1 $ \\
151 & 7 & 5 & $ \bigl(x + 1\bigr)\allowbreak\, \bigl(x^{2} + 115x + 84\bigr)\allowbreak\, \bigl(x^{2} + 129x + 9\bigr) $ \\
157 & 5 & 0 & $ 1 $ \\
163 & 3 & 0 & $ 1 $ \\
167 & 7 & 7 & $ \bigl(x + 1\bigr)\allowbreak\, \bigl(x + 32\bigr)\allowbreak\, \bigl(x + 47\bigr)\allowbreak\, \bigl(x^{2} + 17x + 1\bigr)\allowbreak\, \bigl(x^{2} + 19x + 1\bigr) $ \\
173 & 5 & 0 & $ 1 $ \\
179 & 3 & 0 & $ 1 $ \\
181 & 5 & 0 & $ 1 $ \\
191 & 7 & 13 & $ \bigl(x + 1\bigr)\allowbreak\, \bigl(x + 101\bigr)\allowbreak\, \bigl(x + 106\bigr)\allowbreak\, \bigl(x + 126\bigr)\allowbreak\, \bigl(x + 141\bigr)\allowbreak\, \bigl(x + 182\bigr)\allowbreak\, \bigl(x + 42\bigr)\allowbreak\, \bigl(x + 47\bigr)\allowbreak\, \bigl(x + 87\bigr)\allowbreak\, \bigl(x^{2} + 167x + 1\bigr)\allowbreak\, \bigl(x^{2} + x + 1\bigr) $ \\
193 & 1 & 2 & $ \bigl(x^{2} + 184x + 1\bigr) $ \\
197 & 5 & 0 & $ 1 $ \\
199 & 7 & 3 & $ \bigl(x + 1\bigr)\allowbreak\, \bigl(x + 7\bigr)\allowbreak\, \bigl(x + 57\bigr) $ \\
211 & 3 & 0 & $ 1 $ \\
223 & 7 & 3 & $ \bigl(x + 1\bigr)\allowbreak\, \bigl(x + 163\bigr)\allowbreak\, \bigl(x + 26\bigr) $ \\
227 & 3 & 0 & $ 1 $ \\
229 & 5 & 0 & $ 1 $ \\
233 & 1 & 0 & $ 1 $ \\
239 & 7 & 7 & $ \bigl(x + 1\bigr)\allowbreak\, \bigl(x + 178\bigr)\allowbreak\, \bigl(x + 47\bigr)\allowbreak\, \bigl(x^{2} + 238x + 1\bigr)\allowbreak\, \bigl(x^{2} + 25x + 1\bigr) $ \\
241 & 1 & 0 & $ 1 $ \\
251 & 3 & 0 & $ 1 $ \\
257 & 1 & 10 & $ \bigl(x^{2} + 254x + 1\bigr)\allowbreak\, \bigl(x^{2} + 256x + 16\bigr)\allowbreak\, \bigl(x^{2} + 16x + 241\bigr)\allowbreak\, \bigl(x^{2} + 63x + 129\bigr)\allowbreak\, \bigl(x^{2} + 126x + 2\bigr) $ \\
263 & 7 & 5 & $ \bigl(x + 1\bigr)\allowbreak\, \bigl(x^{2} + 216x + 1\bigr)\allowbreak\, \bigl(x^{2} + 231x + 1\bigr) $ \\
269 & 5 & 0 & $ 1 $ \\
271 & 7 & 7 & $ \bigl(x + 1\bigr)\allowbreak\, \bigl(x + 244\bigr)\allowbreak\, \bigl(x + 10\bigr)\allowbreak\, \bigl(x + 26\bigr)\allowbreak\, \bigl(x + 73\bigr)\allowbreak\, \bigl(x^{2} + 15x + 1\bigr) $ \\
277 & 5 & 0 & $ 1 $ \\
281 & 1 & 0 & $ 1 $ \\
283 & 3 & 0 & $ 1 $ \\
293 & 5 & 0 & $ 1 $ \\
307 & 3 & 0 & $ 1 $ \\
311 & 7 & 7 & $ \bigl(x + 1\bigr)\allowbreak\, \bigl(x^{2} + 176x + 1\bigr)\allowbreak\, \bigl(x^{2} + 71x + 1\bigr)\allowbreak\, \bigl(x^{2} + 73x + 1\bigr) $ \\
313 & 1 & 4 & $ \bigl(x^{2} + 278x + 137\bigr)\allowbreak\, \bigl(x^{2} + 66x + 16\bigr) $ \\
317 & 5 & 0 & $ 1 $ \\
331 & 3 & 0 & $ 1 $ \\
337 & 1 & 6 & $ \bigl(x^{2} + 199x + 8\bigr)\allowbreak\, \bigl(x^{2} + 206x + 1\bigr)\allowbreak\, \bigl(x^{2} + 67x + 295\bigr) $ \\
347 & 3 & 0 & $ 1 $ \\
349 & 5 & 0 & $ 1 $ \\
353 & 1 & 8 & $ \bigl(x^{2} + 335x + 1\bigr)\allowbreak\, \bigl(x^{2} + 350x + 1\bigr)\allowbreak\, \bigl(x^{2} + 36x + 131\bigr)\allowbreak\, \bigl(x^{2} + 38x + 256\bigr) $ \\
359 & 7 & 7 & $ \bigl(x + 1\bigr)\allowbreak\, \bigl(x + 50\bigr)\allowbreak\, \bigl(x + 79\bigr)\allowbreak\, \bigl(x^{2} + 292x + 1\bigr)\allowbreak\, \bigl(x^{2} + 295x + 1\bigr) $ \\
367 & 7 & 3 & $ \bigl(x + 1\bigr)\allowbreak\, \bigl(x + 291\bigr)\allowbreak\, \bigl(x + 169\bigr) $ \\
373 & 5 & 0 & $ 1 $ \\
379 & 3 & 0 & $ 1 $ \\
383 & 7 & 13 & $ \bigl(x + 1\bigr)\allowbreak\, \bigl(x + 245\bigr)\allowbreak\, \bigl(x + 247\bigr)\allowbreak\, \bigl(x + 254\bigr)\allowbreak\, \bigl(x + 374\bigr)\allowbreak\, \bigl(x + 381\bigr)\allowbreak\, \bigl(x + 85\bigr)\allowbreak\, \bigl(x + 95\bigr)\allowbreak\, \bigl(x + 191\bigr)\allowbreak\, \bigl(x^{2} + 236x + 1\bigr)\allowbreak\, \bigl(x^{2} + 246x + 1\bigr) $ \\
389 & 5 & 0 & $ 1 $ \\
397 & 5 & 0 & $ 1 $ \\
401 & 1 & 4 & $ \bigl(x^{2} + 227x + 318\bigr)\allowbreak\, \bigl(x^{2} + 234x + 372\bigr) $ \\
409 & 1 & 2 & $ \bigl(x^{2} + 270x + 1\bigr) $ \\
419 & 3 & 0 & $ 1 $ \\
421 & 5 & 0 & $ 1 $ \\
431 & 7 & 11 & $ \bigl(x + 1\bigr)\allowbreak\, \bigl(x + 219\bigr)\allowbreak\, \bigl(x + 112\bigr)\allowbreak\, \bigl(x + 127\bigr)\allowbreak\, \bigl(x + 148\bigr)\allowbreak\, \bigl(x + 166\bigr)\allowbreak\, \bigl(x + 185\bigr)\allowbreak\, \bigl(x^{2} + 295x + 1\bigr)\allowbreak\, \bigl(x^{2} + 430x + 1\bigr) $ \\
433 & 1 & 10 & $ \bigl(x^{2} + 328x + 78\bigr)\allowbreak\, \bigl(x^{2} + 355x + 51\bigr)\allowbreak\, \bigl(x^{2} + 406x + 17\bigr)\allowbreak\, \bigl(x^{2} + 415x + 161\bigr)\allowbreak\, \bigl(x^{2} + 430x + 1\bigr) $ \\
439 & 7 & 13 & $ \bigl(x + 1\bigr)\allowbreak\, \bigl(x + 10\bigr)\allowbreak\, \bigl(x + 44\bigr)\allowbreak\, \bigl(x + 109\bigr)\allowbreak\, \bigl(x + 145\bigr)\allowbreak\, \bigl(x^{2} + 245x + 1\bigr)\allowbreak\, \bigl(x^{2} + 419x + 286\bigr)\allowbreak\, \bigl(x^{2} + 3x + 373\bigr)\allowbreak\, \bigl(x^{2} + 5x + 1\bigr) $ \\
443 & 3 & 0 & $ 1 $ \\
449 & 1 & 4 & $ \bigl(x^{2} + 418x + 424\bigr)\allowbreak\, \bigl(x^{2} + 109x + 431\bigr) $ \\
457 & 1 & 4 & $ \bigl(x^{2} + 424x + 262\bigr)\allowbreak\, \bigl(x^{2} + 190x + 382\bigr) $ \\
461 & 5 & 0 & $ 1 $ \\
463 & 7 & 1 & $ \bigl(x + 1\bigr) $ \\
467 & 3 & 0 & $ 1 $ \\
479 & 7 & 11 & $ \bigl(x + 1\bigr)\allowbreak\, \bigl(x + 254\bigr)\allowbreak\, \bigl(x + 265\bigr)\allowbreak\, \bigl(x + 311\bigr)\allowbreak\, \bigl(x + 342\bigr)\allowbreak\, \bigl(x + 413\bigr)\allowbreak\, \bigl(x + 449\bigr)\allowbreak\, \bigl(x + 463\bigr)\allowbreak\, \bigl(x + 472\bigr)\allowbreak\, \bigl(x + 47\bigr)\allowbreak\, \bigl(x + 134\bigr) $ \\
487 & 7 & 1 & $ \bigl(x + 1\bigr) $ \\
491 & 3 & 0 & $ 1 $ \\
499 & 3 & 0 & $ 1 $ \\
503 & 7 & 1 & $ \bigl(x + 1\bigr) $ \\
509 & 5 & 0 & $ 1 $ \\
521 & 1 & 0 & $ 1 $ \\
523 & 3 & 0 & $ 1 $ \\
541 & 5 & 0 & $ 1 $ \\
547 & 3 & 0 & $ 1 $ \\
557 & 5 & 0 & $ 1 $ \\
563 & 3 & 0 & $ 1 $ \\
569 & 1 & 8 & $ \bigl(x^{2} + 294x + 399\bigr)\allowbreak\, \bigl(x^{2} + 416x + 153\bigr)\allowbreak\, \bigl(x^{2} + 420x + 164\bigr)\allowbreak\, \bigl(x^{2} + 568x + 450\bigr) $ \\
571 & 3 & 0 & $ 1 $ \\
577 & 1 & 6 & $ \bigl(x^{2} + 521x + 24\bigr)\allowbreak\, \bigl(x^{2} + 539x + 1\bigr)\allowbreak\, \bigl(x^{2} + 190x + 553\bigr) $ \\
587 & 3 & 0 & $ 1 $ \\
593 & 1 & 6 & $ \bigl(x^{2} + 568x + 1\bigr)\allowbreak\, \bigl(x^{2} + 577x + 16\bigr)\allowbreak\, \bigl(x^{2} + 592x + 556\bigr) $ \\
599 & 7 & 9 & $ \bigl(x + 1\bigr)\allowbreak\, \bigl(x + 511\bigr)\allowbreak\, \bigl(x + 211\bigr)\allowbreak\, \bigl(x^{2} + 333x + 1\bigr)\allowbreak\, \bigl(x^{2} + 406x + 1\bigr)\allowbreak\, \bigl(x^{2} + 121x + 1\bigr) $ \\
601 & 1 & 4 & $ \bigl(x^{2} + 449x + 453\bigr)\allowbreak\, \bigl(x^{2} + 66x + 134\bigr) $ \\
607 & 7 & 15 & $ \bigl(x + 1\bigr)\allowbreak\, \bigl(x + 348\bigr)\allowbreak\, \bigl(x + 385\bigr)\allowbreak\, \bigl(x + 390\bigr)\allowbreak\, \bigl(x + 400\bigr)\allowbreak\, \bigl(x + 532\bigr)\allowbreak\, \bigl(x + 605\bigr)\allowbreak\, \bigl(x + 216\bigr)\allowbreak\, \bigl(x + 303\bigr)\allowbreak\, \bigl(x^{2} + 466x + 1\bigr)\allowbreak\, \bigl(x^{2} + 164x + 60\bigr)\allowbreak\, \bigl(x^{2} + 286x + 172\bigr) $ \\
613 & 5 & 0 & $ 1 $ \\
617 & 1 & 0 & $ 1 $ \\
619 & 3 & 0 & $ 1 $ \\
631 & 7 & 7 & $ \bigl(x + 1\bigr)\allowbreak\, \bigl(x + 31\bigr)\allowbreak\, \bigl(x + 285\bigr)\allowbreak\, \bigl(x^{2} + 348x + 496\bigr)\allowbreak\, \bigl(x^{2} + 362x + 215\bigr) $ \\
641 & 1 & 0 & $ 1 $ \\
643 & 3 & 0 & $ 1 $ \\
647 & 7 & 3 & $ \bigl(x + 1\bigr)\allowbreak\, \bigl(x^{2} + 165x + 1\bigr) $ \\
653 & 5 & 0 & $ 1 $ \\
659 & 3 & 0 & $ 1 $ \\
661 & 5 & 0 & $ 1 $ \\
673 & 1 & 4 & $ \bigl(x^{2} + 358x + 256\bigr)\allowbreak\, \bigl(x^{2} + 238x + 418\bigr) $ \\
677 & 5 & 0 & $ 1 $ \\
683 & 3 & 0 & $ 1 $ \\
691 & 3 & 0 & $ 1 $ \\
701 & 5 & 0 & $ 1 $ \\
709 & 5 & 0 & $ 1 $ \\
719 & 7 & 17 & $ \bigl(x + 1\bigr)\allowbreak\, \bigl(x + 373\bigr)\allowbreak\, \bigl(x + 603\bigr)\allowbreak\, \bigl(x + 611\bigr)\allowbreak\, \bigl(x + 641\bigr)\allowbreak\, \bigl(x + 667\bigr)\allowbreak\, \bigl(x + 688\bigr)\allowbreak\, \bigl(x + 212\bigr)\allowbreak\, \bigl(x + 233\bigr)\allowbreak\, \bigl(x + 293\bigr)\allowbreak\, \bigl(x + 318\bigr)\allowbreak\, \bigl(x^{2} + 438x + 652\bigr)\allowbreak\, \bigl(x^{2} + 565x + 1\bigr)\allowbreak\, \bigl(x^{2} + 691x + 279\bigr) $ \\
727 & 7 & 7 & $ \bigl(x + 1\bigr)\allowbreak\, \bigl(x^{2} + 675x + 157\bigr)\allowbreak\, \bigl(x^{2} + 31x + 1\bigr)\allowbreak\, \bigl(x^{2} + 171x + 514\bigr) $ \\
733 & 5 & 0 & $ 1 $ \\
739 & 3 & 0 & $ 1 $ \\
743 & 7 & 5 & $ \bigl(x + 1\bigr)\allowbreak\, \bigl(x^{2} + 679x + 539\bigr)\allowbreak\, \bigl(x^{2} + 146x + 346\bigr) $ \\
751 & 7 & 1 & $ \bigl(x + 1\bigr) $ \\
757 & 5 & 0 & $ 1 $ \\
761 & 1 & 10 & $ \bigl(x^{2} + 760x + 728\bigr)\allowbreak\, \bigl(x^{2} + 760x + 132\bigr)\allowbreak\, \bigl(x^{2} + 98x + 663\bigr)\allowbreak\, \bigl(x^{2} + 355x + 1\bigr)\allowbreak\, \bigl(x^{2} + 369x + 392\bigr) $ \\
769 & 1 & 2 & $ \bigl(x^{2} + 450x + 1\bigr) $ \\
773 & 5 & 0 & $ 1 $ \\
787 & 3 & 0 & $ 1 $ \\
797 & 5 & 0 & $ 1 $ \\
809 & 1 & 4 & $ \bigl(x^{2} + 475x + 349\bigr)\allowbreak\, \bigl(x^{2} + 764x + 51\bigr) $ \\
811 & 3 & 0 & $ 1 $ \\
821 & 5 & 0 & $ 1 $ \\
823 & 7 & 5 & $ \bigl(x + 1\bigr)\allowbreak\, \bigl(x^{2} + 422x + 1\bigr)\allowbreak\, \bigl(x^{2} + 437x + 1\bigr) $ \\
827 & 3 & 0 & $ 1 $ \\
829 & 5 & 0 & $ 1 $ \\
839 & 7 & 13 & $ \bigl(x + 1\bigr)\allowbreak\, \bigl(x^{2} + 486x + 210\bigr)\allowbreak\, \bigl(x^{2} + 605x + 1\bigr)\allowbreak\, \bigl(x^{2} + 13x + 1\bigr)\allowbreak\, \bigl(x^{2} + 219x + 1\bigr)\allowbreak\, \bigl(x^{2} + 264x + 1\bigr)\allowbreak\, \bigl(x^{2} + 266x + 4\bigr) $ \\
853 & 5 & 0 & $ 1 $ \\
857 & 1 & 14 & $ \bigl(x^{2} + 599x + 377\bigr)\allowbreak\, \bigl(x^{2} + 653x + 323\bigr)\allowbreak\, \bigl(x^{2} + 654x + 416\bigr)\allowbreak\, \bigl(x^{2} + 783x + 1\bigr)\allowbreak\, \bigl(x^{2} + 856x + 488\bigr)\allowbreak\, \bigl(x^{2} + 72x + 785\bigr)\allowbreak\, \bigl(x^{2} + 270x + 528\bigr) $ \\
859 & 3 & 0 & $ 1 $ \\
863 & 7 & 11 & $ \bigl(x + 1\bigr)\allowbreak\, \bigl(x + 439\bigr)\allowbreak\, \bigl(x + 480\bigr)\allowbreak\, \bigl(x + 633\bigr)\allowbreak\, \bigl(x + 768\bigr)\allowbreak\, \bigl(x + 10\bigr)\allowbreak\, \bigl(x + 109\bigr)\allowbreak\, \bigl(x + 259\bigr)\allowbreak\, \bigl(x + 347\bigr)\allowbreak\, \bigl(x^{2} + 783x + 1\bigr) $ \\
877 & 5 & 0 & $ 1 $ \\
881 & 1 & 6 & $ \bigl(x^{2} + 727x + 333\bigr)\allowbreak\, \bigl(x^{2} + 235x + 336\bigr)\allowbreak\, \bigl(x^{2} + 254x + 1\bigr) $ \\
883 & 3 & 0 & $ 1 $ \\
887 & 7 & 1 & $ \bigl(x + 1\bigr) $ \\
907 & 3 & 0 & $ 1 $ \\
911 & 7 & 11 & $ \bigl(x + 1\bigr)\allowbreak\, \bigl(x + 483\bigr)\allowbreak\, \bigl(x + 784\bigr)\allowbreak\, \bigl(x + 221\bigr)\allowbreak\, \bigl(x + 371\bigr)\allowbreak\, \bigl(x + 373\bigr)\allowbreak\, \bigl(x + 381\bigr)\allowbreak\, \bigl(x^{2} + 707x + 1\bigr)\allowbreak\, \bigl(x^{2} + 226x + 1\bigr) $ \\
919 & 7 & 1 & $ \bigl(x + 1\bigr) $ \\
929 & 1 & 0 & $ 1 $ \\
937 & 1 & 6 & $ \bigl(x^{2} + 65x + 256\bigr)\allowbreak\, \bigl(x^{2} + 143x + 377\bigr)\allowbreak\, \bigl(x^{2} + 242x + 1\bigr) $ \\
941 & 5 & 0 & $ 1 $ \\
947 & 3 & 0 & $ 1 $ \\
953 & 1 & 10 & $ \bigl(x^{2} + 481x + 1\bigr)\allowbreak\, \bigl(x^{2} + 586x + 872\bigr)\allowbreak\, \bigl(x^{2} + 679x + 209\bigr)\allowbreak\, \bigl(x^{2} + 934x + 200\bigr)\allowbreak\, \bigl(x^{2} + 213x + 114\bigr) $ \\
967 & 7 & 3 & $ \bigl(x + 1\bigr)\allowbreak\, \bigl(x + 544\bigr)\allowbreak\, \bigl(x + 16\bigr) $ \\
971 & 3 & 0 & $ 1 $ \\
977 & 1 & 0 & $ 1 $ \\
983 & 7 & 5 & $ \bigl(x + 1\bigr)\allowbreak\, \bigl(x + 85\bigr)\allowbreak\, \bigl(x + 266\bigr)\allowbreak\, \bigl(x^{2} + 90x + 1\bigr) $ \\
991 & 7 & 3 & $ \bigl(x + 1\bigr)\allowbreak\, \bigl(x + 976\bigr)\allowbreak\, \bigl(x + 66\bigr) $ \\
997 & 5 & 0 & $ 1 $ \\
\end{longtable}
\endgroup


\begin{thebibliography}{vdGvdV}

\bibitem[1]{AchterHowe}
J. D. Achter and E. W. Howe,
\textit{Hasse--Witt and Cartier--Manin matrices: A warning and a request},
in \textit{Arithmetic Geometry: Computation and Applications},
Contemp. Math. 722, Amer. Math. Soc., Providence, RI, 2019, pp. 1--18.

\bibitem[2]{AuerTop}
R. Auer and J. Top,
\textit{Legendre elliptic curves over finite fields},
J. Number Theory 95 (2002), no. 2, 303--312.


\bibitem[3]{IbukiyamaKatsuraOort}
T. Ibukiyama, T. Katsura, and F. Oort,
\textit{Supersingular curves of genus two and class numbers},
Compositio Math. 57 (1986), no. 2, 127--152.

\bibitem[4]{Igusa}
J.-I. Igusa,
\textit{On the algebraic theory of elliptic modular functions},
J. Math. Soc. Japan 20 (1968), no. 1--2, 96--106.

\bibitem[5]{KatsuraOort}
T. Katsura and F. Oort,
\textit{Families of supersingular abelian surfaces},
Compositio Math. 62 (1987), no. 2, 107--167.

\bibitem[6]{Nygaard}
N. O. Nygaard,
\textit{Slopes of powers of Frobenius on crystalline cohomology},
Ann. Sci. \'Ecole Norm. Sup. (4) 14 (1981), no. 4, 369--401.

\bibitem[7]{Ohashi}
R. Ohashi,
\textit{On the maximality of hyperelliptic Howe curves of genus 3},
Kodai Math. J. 45 (2022), no. 2, 282--294.

\bibitem[8]{OhashiKudo}
R. Ohashi and M. Kudo,
\textit{Computing superspecial hyperelliptic curves of genus $4$ with
automorphism group properly containing the Klein $4$-group},
J. Comput. Algebra 11 (2024), Paper No. 100020.

\bibitem[9]{Oort}
F. Oort,
\textit{Which abelian surfaces are products of elliptic curves?},
Math. Ann. 214 (1975), 35--47.

\bibitem[10]{GeerVlugt}
G. van der Geer and M. van der Vlugt,
\textit{Reed--Muller codes and supersingular curves. I},
Compositio Math. 84 (1992), no. 3, 333--367.

\end{thebibliography}
\end{document}